\documentclass[14pt]{amsart}

\usepackage{amsmath,amssymb,amsfonts,enumerate,amsthm, amscd}
\usepackage[french, ngerman, italian, english]{babel}
\usepackage[colorlinks=true]{hyperref}
\usepackage{paralist}
\usepackage{verbatim}
\usepackage{color}
\usepackage{xcolor}
\usepackage{tikz}
\usetikzlibrary{arrows,decorations.pathmorphing,backgrounds,positioning,fit}
\RequirePackage{pgfrcs}

\theoremstyle{plain}
\newtheorem{thm}{\bf Theorem}[section]

\newtheorem{Mainthm}{\bf Theorem}
\newtheorem{prop}[thm]{\bf Proposition}

\newtheorem{lem}[thm]{\bf Lemma}
\newtheorem{cor}[thm]{\bf Corollary}

\theoremstyle{definition}
\newtheorem{defn}[thm]{\bf Definition}
\theoremstyle{remark}
\newtheorem{rem}[thm]{\bf Remark}

\theoremstyle{example}

\def \supp{{\mathrm{supp}}}

\def \a{\mathbf a}

\def \1{\mathbf 1}
\def \c{\mathrm{core}}

\def \reg{\mathrm{reg}}

\def \depth{\mathrm{depth}}

\def \pd{\mathrm{pd}}
\def \bight{\mathrm{bight}}

\def \star{\mathrm{star}}

\def \lk{\mathrm{link}}
\def \Del{\Delta}
\def \set{\mathrm{set}}
\def \NN{\mathbb N}
\def \ZZ{\mathbb Z}

\def \fm{\mathfrak m}
\def \fn{\mathfrak n}

\def \F{\mathcal F}

\def \C{\mathcal C}

\begin{document}

\title[On the Stanley-Reisner ideal of an expanded simplicial complex]{On the Stanley-Reisner ideal of an expanded simplicial complex}
\author{Rahim Rahmati-Asghar and Somayeh Moradi}



\begin{abstract}
 Let $\Delta$ be  a simplicial complex.  We study the expansions of $\Delta$ mainly to see how the algebraic and combinatorial properties of $\Delta$ and its expansions are related to each other. It is shown that $\Delta$ is Cohen-Macaulay, sequentially Cohen-Macaulay, Buchsbaum or $k$-decomposable, if and only if an arbitrary expansion of $\Delta$ has the same property. Moreover, some homological invariants like the regularity and the projective dimension of the Stanley-Reisner ideals of $\Delta$ and those of their expansions are compared.
\end{abstract}

\maketitle

\section*{Introduction}
Let $S=K[x_1,\ldots,x_n]$ be a polynomial ring over a field $K$.
Any squarefree monomial ideal $I$ in $S$ can be considered both as the Stanley-Reisner ideal of a simplicial complex $\Delta_I=\{\{x_1,\ldots,x_n\}:\ x_1\cdots x_n\notin I\}$  and as a facet ideal of the simplicial complex $\Delta'=\langle F:\ x^F \textrm{ is a minimal generator of } I\rangle$.
Each of these considerations make a natural one-to-one correspondence
between the class of squarefree
monomial ideals in $S$ and the class of simplicial complexes on $\{x_1,\ldots,x_n\}$. Thus simplicial complexes play an important role in the study monomial ideals. In this regard classifying simplicial complexes with a desired property or making modifications to a structure like a graph or a simplicial complex so that it satisfies a special property, has been considered in many research papers, see for example \cite{BaHe,ER,F,FH,FV,HMV,KhMo,Wo}.

The notion of expansion of a simplicial complex was defined in \cite{KhMo} as a natural generalization of the concept of expansion in graph theory and some properties of a simplicial complex and its expansions were related to each other.
Our goal in this paper is to investigate more relations between algebraic properties of the Stanley-Reisner ideal of a simplicial complex and those of its expansions, which generalize the results proved in \cite{KhMo}. It turns out that many algebraic and combinatorial properties of a simplicial complex and its expansions are equivalent and so this construction is a very good tool to make new simplicial complexes with a desired property.

The paper is organized as follows. In the first section, we review some preliminaries which are needed in the sequel.
In Section 2, we study the Stanley-Reisner ideal of an expanded complex. One of the main results is the following theorem.

\begin{Mainthm}(Theorem \ref{CM})
Let $\Del$ be a simplicial complex and $\alpha\in\NN^n$. Then $\Delta$ is Cohen-Macaulay if and only if $\Delta^{\alpha}$ is Cohen-Macaulay.
\end{Mainthm}

As a corollary, it is shown that sequentially Cohen-Macaulayness  of a simplicial complex and its expansion are also equivalent.

Moreover, using an epimorphism which relates the reduced simplicial homology groups,  we show that an expansion of $\Delta$ is Buchsbaum if and only if $\Delta$ has the same property (see Theorem \ref{Buchs}). Theorem \ref{decom} (resp. Corollary \ref{shell}) shows that $\Delta$ is  $k$-decomposable (resp. shellable) if and only if an arbitrary expansion of $\Delta$ is $k$-decomposable (resp. shellable). Section 3 is devoted to studying some homological invariants of the Stanley-Reisner ideals of $\Delta$ and $\Del^\alpha$. In fact, we give inequalities which relate the regularity, the projective dimension and the depth of a simplicial complex to those of its expansion.

\section{Preliminaries}

Throughout this paper, we assume that $\Delta$ is a simplicial complex on the vertex set $X=\{x_1, \dots, x_n\}$,
$K$ is a field and  $S=K[X]$ is a polynomial ring.
The set of facets (maximal faces) of $\Delta$
is denoted by $\mathcal{F}(\Delta)$ and if $\mathcal{F}(\Delta)=\{F_1,\ldots,F_r\}$, we write $\Delta=\langle F_1,\ldots,F_r\rangle$. For a monomial ideal $I$ of $S$, the set of minimal generators of $I$ is denoted by $\mathcal{G}(I)$.
For $\alpha=(s_1,\ldots,s_n)\in\NN^n$, we set $X^{\alpha}=\{x_{11},\ldots,x_{1s_1},\ldots,x_{n1},\ldots,x_{ns_n}\}$ and $S^\alpha=K[X^{\alpha}].$

The concept of expansion of a simplicial complex was defined in \cite{KhMo} as follows.

\begin{defn}
Let $\Del$ be a simplicial complex on $X$, $\alpha=(s_1,\ldots,s_n)\in\NN^n$ and $F=\{x_{i_1},\ldots ,x_{i_r}\}$  be a facet of $\Del$. The \textbf{expansion} of the simplex $\langle F\rangle$ with respect to $\alpha$ is denoted by $\langle F\rangle^\alpha$ and is defined as a simplicial complex on the vertex set $\{x_{i_lt_l}:1\leq l\leq r,\ 1\leq t_l\leq s_{i_l}\}$ with facets $$\{\{x_{i_1j_1},\ldots, x_{i_rj_r}\}: 1\leq j_m\leq s_{i_m}\}.$$
The expansion of $\Del$ with respect to $\alpha$ is defined as $$\Del^\alpha=\bigcup_{F\in\Del}\langle F\rangle^\alpha.$$
A simplicial complex obtained by an expansion, is called an expanded complex.
\end{defn}

\begin{defn}
{\rm A simplicial complex $\Delta$ is called \textbf{shellable} if there exists an ordering $F_1<\cdots<F_m$ on the
facets of $\Delta$
such that for any $i<j$, there exists a vertex
$v\in F_j\setminus F_i$ and  $\ell<j$ with
$F_j\setminus F_\ell=\{v\}$. We call $F_1,\ldots,F_m$ a \textbf{shelling} for
$\Delta$.}
\end{defn}

For a simplicial complex $\Delta$ and $F\in \Delta$, the link of $F$ in
$\Delta$ is defined as $$\lk_{\Delta}(F)=\{G\in \Delta: G\cap
F=\emptyset, G\cup F\in \Delta\},$$ and the deletion of $F$ is the
simplicial complex $$\Delta \setminus F=\{G\in \Delta: F \nsubseteq G\}.$$

Woodroofe in \cite{Wo} extended the definition of $k$-decomposability
to non-pure complexes as follows.

Let $\Delta$ be a simplicial complex on vertex set $X$. Then a face $\sigma$ is called a
\textbf{shedding face} if every face $\tau$  containing $\sigma$ satisfies the following exchange property: for every
$v \in \sigma$ there is $w\in X \setminus \tau$ such that $(\tau \cup \{w\})\setminus \{v\}$ is a face of $\Delta$.

\begin{defn}\cite[Definition 3.5]{Wo}
{\rm A simplicial complex $\Delta$ is recursively defined to be \textbf{$k$-decomposable} if
either $\Delta$ is a simplex or else has a shedding face $\sigma$ with $\dim(\sigma)\leq k$ such that both $\Delta \setminus \sigma$
and $\lk_{\Delta}(\sigma)$ are $k$-decomposable. The complexes $\{\}$ and $\{\emptyset\}$  are considered to be
$k$-decomposable for all $k \geq -1$.}
\end{defn}

Note that $0$-decomposable simplicial complexes are precisely vertex decomposable simplicial complexes.

Also the notion of a decomposable monomial ideal was introduced in \cite{RaYa} as follows.
For the monomial $u=x_1^{a_1}\cdots x_n^{a_n}$ in $S$, the support of $u$  denoted by $\supp(u)$ is the set $\{x_i:\ a_i\neq 0\}$. For a monomial $M$ in $S$,
set $[u,M] = 1$ if for all $x_i\in\supp(u)$, $x_i^{a_i}\nmid M$. Otherwise set $[u,M]\neq 1$.

For the monomial $u$ and the monomial ideal $I$, set
$$I^u = (M\in \mathcal{G}(I) :\ [u,M]\neq 1)$$ and $$I_u = (M\in \mathcal{G}(I) :\ [u,M]=1).$$
 For a monomial ideal $I$  with $\mathcal{G}(I)=\{M_1,\ldots,M_r\}$, the monomial $u=x_1^{a_1}
\cdots x_n^{a_n}$ is called  a \textbf{shedding monomial} for $I$ if $I_u\neq 0$ and for each $M_i\in \mathcal{G}(I_u)$ and each
$x_l\in \supp(u)$ there exists $M_j\in \mathcal{G}(I^u)$ such that $M_j:M_i=x_l$.

\begin{defn}\cite[Definition 2.3]{RaYa}
{\rm
A monomial ideal $I$ with $\mathcal{G}(I)=\{M_1,\ldots,M_r\}$ is called  \textbf{$k$-decomposable} if $r=1$ or else has a shedding monomial $u$ with
$|\supp(u)|\leq k+1$ such that the ideals $I_u$ and $I^u$ are $k$-decomposable.}
\end{defn}

\begin{defn}\label{1.2}
{\rm
A monomial ideal $I$ in the ring $S$ has \textbf{linear quotients} if there exists an ordering $f_1, \dots, f_m$ on the minimal generators of $I$ such that the colon ideal $(f_1,\ldots,f_{i-1}):(f_i)$ is generated by a subset of $\{x_1,\ldots,x_n\}$ for all $2\leq i\leq m$. We show this ordering by $f_1<\dots <f_m$ and we call it an order of linear quotients on $\mathcal{G}(I)$.

Let $I$ be a monomial ideal which has linear quotients and $f_1<\dots <f_m$ be an order of linear quotients on the minimal generators of $I$. For any $1\leq i\leq m$, $\set_I(f_i)$ is defined as
$$\set_I(f_i)=\{x_k:\ x_k\in (f_1,\ldots, f_{i-1}) : (f_i)\}.$$
}
\end{defn}

\begin{defn}
{\rm A graded $S$-module $M$ is called
\textbf{sequentially Cohen--Macaulay} (over a field $K$) if there exists a
finite filtration of graded $S$-modules $$0=M_0\subset M_1\subset
\cdots \subset M_r=M$$ such that each $M_i/M_{i-1}$ is
Cohen--Macaulay and
$$\dim(M_1/M_0)<\dim(M_2/M_1)<\cdots<\dim(M_r/M_{r-1}).$$}
\end{defn}

For a $\mathbb{Z}$-graded $S$-module $M$, the \textbf{Castelnuovo-Mumford regularity} (or briefly regularity)
of $M$ is defined as
$$\reg(M) = \max\{j-i: \ \beta_{i,j}(M)\neq 0\},$$
and the \textbf{projective dimension} of $M$ is defined as
$$\pd(M) = \max\{i:\ \beta_{i,j}(M)\neq 0 \ \text{for some}\ j\},$$
where $\beta_{i,j}(M)$ is the $(i,j)$th graded Betti number of $M$.

For a simplicial complex $\Delta$ with the vertex set $X$, the \textbf{Alexander dual simplicial complex} associated to $\Delta$ is defined as
$$\Delta^{\vee}=:\{X\setminus F:\ F\notin \Delta\}.$$

For a squarefree monomial ideal $I=( x_{11}\cdots
x_{1n_1},\ldots,x_{t1}\cdots x_{tn_t})$, the \textbf{Alexander dual ideal} of $I$, denoted by
$I^{\vee}$, is defined as
$$I^{\vee}:=(x_{11},\ldots, x_{1n_1})\cap \cdots \cap (x_{t1},\ldots, x_{tn_t}).$$
For a subset $C\subseteq X$, by $x^C$ we mean the monomial $\prod_{x\in C} x$.
One can see that
$$(I_{\Delta})^{\vee}=(x^{F^c} \ : \ F\in \mathcal{F}(\Delta)), $$
where $I_{\Delta}$ is the Stanley-Reisner ideal associated to $\Delta$ and $F^c=X\setminus F$.
Moreover, $(I_{\Delta})^{\vee}=I_{\Delta^{\vee}}$.

A simplicial complex $\Delta$ is called Cohen-Macaulay (resp. sequentially Cohen-Macaulay, Buchsbaum and Gorenstein), if its the Stanley Reisner ring $K[\Delta]=S/I_{\Delta}$ is Cohen-Macaulay (resp. sequentially Cohen-Macaulay, Buchsbaum and Gorenstein).

For a simplicial complex $\Delta$, the facet ideal of $\Delta$ is defined as $I(\Delta)=(x^F:\ F\in \mathcal{F}(\Delta))$.
Also the complement of $\Delta$ is the simplicial complex $\Delta^c=\langle F^c:\ F\in \mathcal{F}(\Delta)\rangle$. In fact $I(\Delta^c)=I_{\Delta^{\vee}}$.

\section{Algebraic properties of an expanded complex}

In this section, for a simplicial complex $\Delta$, we study the Stanley-Reisner ideal $I_{\Delta^{\alpha}}$ to see how its algebraic properties are related to those of $I_{\Delta}$.

\begin{prop}\label{expansion}
Let $\Del$ be a simplicial complex on $X$ and let $\alpha\in\NN^n$.
\begin{enumerate}[\upshape (i)]
  \item $\dim(\Del^\alpha)=\dim(\Del)$ and $\Del$ is pure if and only if $\Del^\alpha$ is pure;
  \item For $F\in\Del$ and for every facet $G\in\langle F\rangle^\alpha$, we have
$$\lk_{\Del^\alpha}(G)=(\lk_\Del (F))^\alpha;$$
  \item For all $i\leq\dim(\Del)$, there exists an epimorphism $\theta:\tilde{H}_{i}(\Del^\alpha;K)\rightarrow\tilde{H}_{i}(\Del;K)$ and so
  $$\frac{\tilde{H}_{i}(\Del^\alpha;K)}{\ker(\theta)}\cong\tilde{H}_{i}(\Del;K).$$
\end{enumerate}
\end{prop}
\begin{proof}
(i) and (ii) are easily verified.

(iii) Let the map $\varphi:\Del^\alpha\rightarrow\Del$ be defined by $\varphi(\{x_{i_1j_1},\ldots,x_{i_qj_q}\})=\{x_{i_1},\ldots,x_{i_q}\}$. For each $q$, let $\varphi_\#:\tilde{\C}_q(\Del^\alpha;K)\rightarrow\tilde{\C}_q(\Del;K)$ be a homomorphism  defined on the basis elements as follows.
$$\varphi_\#([x_{i_0j_0},\ldots,x_{i_qj_q}])=\left[\varphi(\{x_{i_0j_0}\}),\ldots,\varphi(\{x_{i_qj_q}\})\right].$$
It is clear from the definitions of $\tilde{\C}_q(\Del^\alpha;K)$ and $\tilde{\C}_q(\Del;K)$ that $\varphi_\#$ is well-defined. Also, define $\varphi^\alpha_q:\tilde{H}_{q}(\Del^\alpha;K)\rightarrow\tilde{H}_{q}(\Del;K)$ by
$$\varphi^{\alpha}_q(z+B_q(\Del^\alpha))=\varphi_\#(z)+B_q(\Del).$$

Consider the diagram
\begin{tabbing}
\hskip 50mm$\tilde{\C}_{q+1}(\Del^\alpha;K)$ $\stackrel{\partial^{\alpha}_{q+1}}\longrightarrow$ $\tilde{\C}_q(\Del^\alpha;K)$\\
\hskip 50mm\ \ \ \ \ \ \ \ $\downarrow$ \hskip 21mm  $\downarrow$\\
\hskip 50mm$\tilde{\C}_{q+1}(\Del;K)$\ \  $\stackrel{\partial_{q+1}}\longrightarrow$\ \  $\tilde{\C}_q(\Del;K)$
\end{tabbing}
where $\partial^{\alpha}_{q+1}$ and $\partial_{q+1}$ are the homomorphism of the chain complexes of $\Delta^{\alpha}$ and $\Delta$, respectively and the vertical homomorphisms are $\varphi^{\alpha}_{q+1}$ and $\varphi^{\alpha}_q$.
One can see that for any basis element $\sigma=[x_{i_0j_0},\ldots,x_{i_{q+1}j_{q+1}}]\in \tilde{\C}_{q+1}(\Del^\alpha;K)$,
$\varphi^{\alpha}_q \partial^{\alpha}_{q+1}(\sigma)=\partial_{q+1}\varphi^{\alpha}_{q+1}(\sigma)$.
So the diagram commutes and then $$\varphi^{\alpha}_q(B_q(\Del^\alpha))\subseteq B_q(\Del).$$
Therefore $\varphi^{\alpha}_q$ is a well-defined homomorphism.

It is easy to see that $\varphi^{\alpha}_q$ is surjective, because for $[x_{i_0},\ldots,x_{i_q}]+B_q(\Del)\in \tilde{H}_q(\Del;K)$ we have $$\varphi^{\alpha}_q([x_{i_01},\ldots,x_{i_q1}]+B_q(\Del^\alpha))=[x_{i_0},\ldots,x_{i_q}]+B_q(\Del).$$
\end{proof}

Reisner gave a criterion for the Cohen-Macaulayness of a simplicial complex as follows.
\begin{thm}\cite[Theorem 5.3.5]{villarreal}\label{Reis}
Let $\Delta$ be a simplicial complex. If $K$ is a field,
then the following conditions are equivalent:
\begin{enumerate}[\upshape (i)]
  \item $\Delta$ is Cohen-Macaulay over $K$;
  \item $\tilde{H}_{i}(\lk_{\Delta}(F);K)=0$ for $F\in \Delta$ and $i<\dim(\lk_{\Delta}(F))$.
\end{enumerate}
\end{thm}

For a simplicial complex $\Delta$ with the vertex set $X$ and $x_i\in X$, an expansion of $\Delta$ obtained by duplicating $x_i$
is the simplicial complex with the vertex set $X'=X\cup\{x'_i\}$, where $x'_i$ is a new vertex, defined as follows
 $$\Delta'=\Delta\cup\langle (F\setminus\{x_i\})\cup\{x'_i\}:\ F\in \mathcal{F}(\Delta), x_i\in F\rangle.$$
 In fact $\Delta'=\Delta^{(s_1,\ldots,s_n)}$, where
$s_j=\left\{
\begin{array}{ll}
1  & \hbox{if}\ j\neq i \\
 2 & \hbox{if}\ j=i.
\end{array}
\right.$

\begin{rem}\label{rem1}
Let $\mathcal{L}$ be a property such that for any simplicial complex $\Delta$ with the property $\mathcal{L}$, any expansion of $\Delta$ obtained by duplicating a vertex of $\Delta$ has the property $\mathcal{L}$. Then by induction for a simplicial complex $\Delta$ with the property $\mathcal{L}$ any expansion of $\Delta$ has the property $\mathcal{L}$, since for any $(s_1,\ldots,s_n)\in \NN^n$, $\Delta^{(s_1,\ldots,s_n)}$ is the expansion of $\Delta^{(s_1,\ldots,s_{i-1},s_i-1,s_{i+1},\ldots,s_n)}$ by duplicating $x_{i1}$.
\end{rem}

Now, we come to one of the main results of this paper.

\begin{thm}\label{CM}
Let $\Del$ be a simplicial complex and $\alpha\in\NN^n$. Then $\Delta$ is Cohen-Macaulay if and only if $\Delta^{\alpha}$ is Cohen-Macaulay.
\end{thm}

\begin{proof}
The 'if' part follows from Proposition \ref{expansion} and Theorem \ref{Reis}.
To prove the converse, let $\Delta$ be Cohen-Macaulay. In the light of Remark \ref{rem1}, it is enough to show that any expansion of $\Delta$ obtained by duplication an arbitrary vertex, is Cohen-Macaulay. Let $X$ be the vertex set of $\Delta$, $x_i\in X$ and $\Delta'$ be an expansion of $\Delta$ by duplicating $x_i$. Then $\Delta'=\Delta \cup \langle  (F\setminus\{x_i\})\cup\{x'_i\}:\ F\in \mathcal{F}(\Delta), x_i\in F\rangle$ and
\begin{multline}
  I_{\Delta'^{\vee}}=(x^{(X\cup\{x'_i\})\setminus F}:\ F\in \mathcal{F}(\Delta'))=(x'_ix^{X\setminus F}:\ F\in \mathcal{F}(\Delta))+(x^{(X\cup\{x'_i\})\setminus (F\cup\{x'_i\}\setminus \{x_i\})}: \\
  \ F\in \mathcal{F}(\Delta), x_i\in F)=
(x'_ix^{X\setminus F}:\ F\in \mathcal{F}(\Delta))+(x_ix^{X\setminus F}:\ F\in \mathcal{F}(\Delta), x_i\in F).
\end{multline}

Thus $$I_{\Delta'^{\vee}}=x'_iI_{\Delta^{\vee}}+x_iI_{(\lk_{\Delta}(x_i))^{\vee}}.$$
Since $\Delta$ is Cohen-Macaulay, by \cite[Proposition 5.3.8]{villarreal}, $\lk_{\Delta}(x_i)$ is also Cohen-Macaulay.
So \cite[Theorem 3]{ER} implies that $I_{\Delta^{\vee}}$ and $I_{(\lk_{\Delta}(x_i))^{\vee}}$ have  $(n-d-1)$-linear resolutions, where $d=\dim(\Delta)$.
Therefore $x'_iI_{\Delta^{\vee}}$ and $x_iI_{(\lk_{\Delta}(x_i))^{\vee}}$ have  $(n-d)$-linear resolutions.
Note that $\mathcal{G}(I_{\Delta'^{\vee}})$ is the disjoint union of $\mathcal{G}(x'_iI_{\Delta^{\vee}})$ and $\mathcal{G}(x_iI_{(\lk_{\Delta}(x_i))^{\vee}})$. Now by \cite[Corollary 2.4]{splitting}, $I_{\Delta'^{\vee}}=x'_iI_{\Delta^{\vee}}+x_iI_{(\lk_{\Delta}(x_i))^{\vee}}$ is a Betti splitting and hence by \cite[Corollary 2.2]{splitting},
$$\reg(I_{\Delta'^{\vee}})=\max\{\reg(x'_iI_{\Delta^{\vee}}),\reg(x_iI_{(\lk_{\Delta}(x_i))^{\vee}}),\reg(x'_iI_{\Delta^{\vee}}\cap x_iI_{(\lk_{\Delta}(x_i))^{\vee}})-1\}.$$
As discussed above, $\reg(x'_iI_{\Delta^{\vee}})=\reg(x_iI_{(\lk_{\Delta}(x_i))^{\vee}})=n-d$.
Note that $x'_iI_{\Delta^{\vee}}\cap x_iI_{(\lk_{\Delta}(x_i))^{\vee}}=(x'_ix^{X\setminus F}:\ F\in \mathcal{F}(\Delta))\cap(x_ix^{X\setminus F}:\ F\in \mathcal{F}(\Delta), x_i\in F)=x'_ix_iI_{(\lk_{\Delta}(x_i))^{\vee}}$. So $\reg(x'_iI_{\Delta^{\vee}}\cap x_iI_{(\lk_{\Delta}(x_i))^{\vee}})=\reg(I_{(\lk_{\Delta}(x_i))^{\vee}})+2=n-d-1+2=n-d+1$.
Thus $\reg(I_{\Delta'^{\vee}})=n-d$. Since $I_{\Delta'^{\vee}}$ is homogenous of degree $n-d$, this implies that $I_{\Delta'^{\vee}}$ has a $(n-d)$-linear resolution (see for example \cite[Lemma 5.55]{ME}). Thus using again  \cite[Theorem 3]{ER} implies that $\Delta'$ is Cohen-Macaulay.
\end{proof}

The following theorem compares Buchsbaumness in a simplicial complex and its expansion.

\begin{thm}\label{Buchs}
Let $\Del$ be a simplicial complex on $X$ and let $\alpha\in\NN^n$. Then $\Del$ is Buchsbaum if and only if $\Del^\alpha$ is.
\end{thm}
\begin{proof}
The ``if'' part follows from Proposition \ref{expansion} and \cite[Theorem 8.1]{St}. Let $\Del$ be Buchsbaum. In the light of Remark \ref{rem1}, it is enough to show that any expansion of $\Delta$ obtained by duplication an arbitrary vertex, is Buchsbaum. Let $\Del'$ be the expansion of $\Del$ obtained by duplicating the vertex $x_i$. By \cite[Theorem 4.5]{BiRo}, a simplicial complex $\Del$ is Buchsbaum if and only if $\Del$ is pure and for any vertex $x\in X$, $\lk_\Del(x)$ is Cohen-Macaulay. We use this fact to prove the assertion.

Since $\lk_{\Del'}(x_i)=\lk_{\Del'}(x'_i)=\lk_\Del(x_i)$, it follows from Buchsbaum-ness of $\Del$ that $\lk_{\Del'}(x_i)$ and $\lk_{\Del'}(x'_i)$ are Cohen-Macaulay. Now suppose that $x_j\in X'=X\cup \{x'_i\}$ with $x_j\neq x_i$ and $x_j\neq x'_i$. Then $\lk_{\Del'}(x_j)$ is the expansion of $\lk_\Del(x_j)$ obtained by duplication of $x_i$. Since $\lk_\Del(x_j)$ is Cohen-Macaulay it follows from Theorem \ref{CM} that $\lk_{\Del'}(x_j)$ is Cohen-Macaulay. Therefore the assertion holds.
\end{proof}

Now, we study the sequentially Cohen-Macaulay property in an expanded complex.
For a simplicial complex $\Delta$ and a subcomplex $\Gamma$ of $\Delta$, $\Delta/\Gamma=\{F\in \Delta:\ F\notin \Gamma\}$ is called a \textbf{relative simplicial complex}.
\begin{lem}\label{rel}
Let $\Del$ be a $(d-1)$-dimensional simplicial complex, $\alpha\in\NN^n$ and $\Del_i$ be the subcomplex of $\Delta$ generated by the $i$-dimensional facets of $\Del$. Set
$$\Omega_i=\Del_i/(\Del_i\cap(\cup_{j>i}\Del_j)).$$
Then for all $0\leq i\leq d-1$, $\Omega^\alpha_i=\Del^\alpha_i/ (\Del^\alpha_i\cap(\cup_{j>i}\Del^\alpha_i)).$
\end{lem}
\begin{proof}
The assertion follows from the fact that if $\Gamma_1$ and $\Gamma_2$ are two simplicial complexes on $X$ and $\alpha\in\NN^n$ then $$(\Gamma_1\cup\Gamma_2)^\alpha=\Gamma^\alpha_1\cup\Gamma^\alpha_2,\qquad (\Gamma_1\cap\Gamma_2)^\alpha=\Gamma^\alpha_1\cap\Gamma^\alpha_2.$$
\end{proof}

\begin{thm}\label{Stanley}
(\cite[Proposition 2.10]{St}) Let $\Del$ be a $(d-1)$-dimensional simplicial complex on $X$. Then $\Del$ is sequentially Cohen-Macaulay if and
only if the relative simplicial complexes $\Omega_i$ (defined in Lemma \ref{rel}) are Cohen-Macaulay for $0\leq i\leq d-1$.
\end{thm}

\begin{cor}\label{SCM}
Let $\Del$ be a simplicial complex and let $\alpha\in\NN^n$. Then $\Delta$ is sequentially Cohen-Macaulay if and only if $\Del^\alpha$ is sequentially Cohen-Macaulay.
\end{cor}
\begin{proof}
Combining Theorem \ref{Stanley} and Lemma \ref{rel} with Theorem \ref{CM} we obtain the assertion.
\end{proof}

For $F\in\Del$, set $\star(F)=\{G\in\Del:F\cup G\in\Del\}$. Let $\Gamma_\Del$ be the induced subcomplex of $\Delta$ on the set $\c(X)$ where $\c(X)=\{x_i\in X:\star(\{x_i\})\neq\Del\}$. A combinatorial description of Gorenstein simplicial complexes was given in \cite{St}. It was proved that the simplicial complex $\Del$ is Gorenstein over a field $K$ if and only if for all $F\in\Gamma_\Del$, $\tilde{H}_i(\lk_{\Gamma_\Del}(F);K)\cong K$ if $i=\dim(\lk_{\Gamma_\Del}(F))$ and $\tilde{H}_i(\lk_{\Gamma_\Del}(F);K)$ is vanished if $i<(\lk_{\Gamma_\Del}(F))$ (see Theorem 5.1 of \cite{St}).

For  a face $F\in\Del^\alpha$, we set $\bar{F}=\{x_i: x_{ij}\in F\ \textrm{for some}\ j\}$.

\begin{lem}\label{core}
Let $\Del$ be a simplicial complex with $\Del=\Gamma_\Del$ and let $\alpha\in\NN^n$. Then $\Gamma_{\Del^\alpha}=(\Gamma_\Del)^\alpha$.
\end{lem}
\begin{proof}
Let $F$ be a facet of $\Gamma_{\Del^\alpha}$. Then $F\in\Del^\alpha$ and $F\subset\c(X^\alpha)$. It is clear that $\bar{F}\in\Del= \Gamma_\Del$. So $F\in(\Gamma_\Del)^\alpha$.

Conversely, suppose that $F\in(\Gamma_\Del)^\alpha$. Then $F\in G^\alpha$ for some $G\in \Gamma_\Del$. Hence $G\in\Del$ and $G\subset\c(X)$. This implies that for every $x_i\in G$, $\star(\{x_i\})\neq\Del$. Thus for every $x_i\in G$ and all $1\leq j\leq k_i$, $\star(\{x_{ij}\})\neq\Del^\alpha$. Therefore $F\subset\c(X^\alpha)$. Finally, it follows from $F\in\Del^\alpha$ that $F\in\Gamma_{\Del^\alpha}$.
\end{proof}

\begin{thm}\label{Gor}
Let $\Del$ be a simplicial complex with $\Del=\Gamma_\Del$ and let $\alpha\in\NN^n$. If $\Del^\alpha$ is Gorenstein then $\Del$ is Gorenstein,  too.
\end{thm}
\begin{proof}
The assertion follows from Lemma \ref{core}, Proposition \ref{expansion} and \cite[Theorem 5.1]{St}.
\end{proof}

\begin{rem}
The converse of Theorem \ref{Gor} does not hold in general. To see this fact we first recall a result from \cite{BrHe} which implies the necessary and sufficient conditions for a simplicial complex to be Gorenstein.

By \cite[Theorem 5.6.2]{BrHe} for a simplicial complex $\Del$ with $\Del=\Gamma_\Del$, $\Del$
is Gorenstein over a field $K$ if and only if $\Del$ is an Euler complex which is Cohen-Macaulay over $K$. Recall that the simplicial complex $\Del$ of dimension $d$ is an \textbf{Euler complex} if $\Del$ is pure and $\widetilde{\chi}(\lk_\Del(F))=(-1)^{\dim \lk_\Del(F)}$ for all $F\in\Del$ where $\widetilde{\chi}(\Del)=\sum^{d-1}_{i=0}(-1)^if_i(\Del)-1$.

Now, consider the simplicial complex $\Del=\langle x_1x_2,x_1x_3,x_2x_3\rangle$ on $\{x_1,x_2,x_3\}$ and let $\alpha=(2,1,1)$. Then $\Del=\Gamma_\Del$ and $\Del^\alpha=\Gamma_{\Del^\alpha}$. It is easy to check that $\Del$ is Euler. Also, $\Del$ is Cohen-Macaulay and so it is Gorenstein. On the other hand, $\widetilde{\chi}(\lk_{\Del^\alpha}(\{x_2\}))=2\neq 1=(-1)^{\dim \lk_{\Del^\alpha}(\{x_2\})}$. It follows that $\Del^\alpha$ is not Euler and so it is not Gorenstein.
\end{rem}

The following theorem was proved in \cite{RaYa}.

\begin{thm}\cite[Theorem 2.10]{RaYa}\label{kdual}
A simplicial complex $\Delta$ is  $k$-decomposable if and only if $I_{\Delta^{\vee}}$ is a $k$-decomposable ideal,
where $k\leq\dim(\Delta)$.
\end{thm}

We use the above theorem to prove the following result, which generalizes \cite[Theorem 2.7]{KhMo}.

\begin{thm}\label{decom}
Let $\Delta$ be a simplicial complex and $\alpha\in\NN^n$. The $\Delta$ is $k$-decomposable if and only if $\Delta^{\alpha}$ is $k$-decomposable.
\end{thm}

\begin{proof}
``Only if part'': Considering Remark \ref{rem1}, it is enough to show that an expansion of $\Delta$ obtained by duplicating one vertex, is $k$-decomposable.
Let $X$ be the vertex set of $\Delta$, $x_i\in X$ and $\Delta'$ be an expansion of $\Delta$ by duplicating $x_i$. Then $\Delta'=\Delta \cup \langle  (F\setminus\{x_i\})\cup\{x'_i\}:\ F\in \mathcal{F}(\Delta), x_i\in F\rangle$. As was shown in the proof of Theorem \ref{CM},
$$I_{\Delta'^{\vee}}=x'_iI_{\Delta^{\vee}}+x_iI_{(\lk_{\Delta}(x_i))^{\vee}}.$$
It is easy to see that $(I_{\Delta'^{\vee}})_{x'_i}=x_iI_{(\lk_{\Delta}(x_i))^{\vee}}$   and $(I_{\Delta'^{\vee}})^{x'_i}=x'_iI_{\Delta^{\vee}}$.
Now, let $\Delta$ be $k$-decomposable.
Then by \cite[Proposition 3.7]{Wo}, $\lk_{\Delta}(x_i)$ is $k$-decomposable. Thus \cite[Theorem 2.10, Lemma 2.6]{RaYa} imply that $x'_iI_{\Delta^{\vee}}$ and $x_iI_{(\lk_{\Delta}(x_i))^{\vee}}$ are $k$-decomposable ideals. Also for any minimal generator $x_ix^{X\setminus F}\in (I_{\Delta'^{\vee}})_{x'_i}$,
we have $(x'_ix^{X\setminus F}:x_ix^{X\setminus F})=(x'_i)$. Thus $x'_i$ is a shedding face of $\Delta'$. Clearly $\dim(x'_i)\leq k$. Thus $I_{\Delta'^{\vee}}$ is a $k$-decomposable ideal. So  using again \cite[Theorem 2.10]{RaYa} yields that $\Delta'$ is $k$-decomposable.

``If part'': Let $\Del'$ be an expansion of $\Del$ obtained by duplicating one vertex $x_i$. We first show that if $\Del'$ is $k$-decomposable then $\Del$ is $k$-decomposable, too.

If $\F(\Del)=\{F\}$ then the expansion of $\Del$ obtained by duplicating one vertex $x_i$ is $\Del'=\langle x_i,x'_i\rangle\ast\langle F\backslash x_i\rangle$ and so it is $k$-decomposable,by Proposition 3.8 of \cite{Wo}. Hence suppose that $\Del$ has more than one facet. Let $\sigma$ be a shedding face of $\Del'$ and let $\lk_{\Del'}\sigma$ and $\Del'\backslash\sigma$ are $k$-decomposable. We have two cases:

Case 1. Let $x'_i\in \sigma$. Then $\Del=\Del'\backslash\sigma$ and so $\Del$ is $k$-decomposble. Similarly, if $x_i\in\sigma$ then $\Del'\backslash\sigma$ and $\Del$ are isomorphic and we are done.

Case 2. Let $x_i\not\in\sigma$ and $x'_i\not\in\sigma$. Then $\lk_{\Del'}\sigma$ and $\Del'\backslash\sigma$ are, respectively, the expansions of $\lk_{\Del}\sigma$ and $\Del\backslash\sigma$ obtained by duplicating $x_i$. Therefore it follows from induction that $\lk_{\Del}\sigma$ and $\Del\backslash\sigma$ are $k$-decomposable. Also, it is trivial that $\sigma$ is a shedding face of $\Del$.

Now suppose that $\alpha=(k_1,\ldots,k_n)\in\NN^n$ and $\Del^\alpha$ is $k$-decomposable. Let $k_i>1$ an set $\beta=(k_1,\ldots,k_{i-1},k_i-1,k_{i+1},\ldots,k_n)$. Then $\Del^\alpha$ is the expansion of $\Del^\beta$ and by above assertion, if $\Del^\alpha$ is $k$-decomposable then $\Del^\beta$ is $k$-decomposable, too. In particular, it follows by induction that $\Del$ is $k$-decomposable, as desired.
\end{proof}

The following theorem which was proved in \cite{Wo}, relates the shellability of a simplicial complex to $d$-decomposability of it.

\begin{thm}\cite[Theorem 3.6]{Wo}\label{decomshell}
A $d$-dimensional simplicial complex
$\Delta$ is shellable if and only if it is $d$-decomposable.
\end{thm}

Using Theorems \ref{decom} and \ref{decomshell}, we get the following corollaries.
\begin{cor}\label{shell}
Let $\Del$ be a simplicial complex and $\alpha\in\NN^n$. Then $\Del$ is shellable if and only if $\Del^\alpha$ is shellable.
\end{cor}

\begin{cor}
Let $\Del$ be a simplicial complex and $\alpha\in\NN^n$. Then $I_\Del$ is clean if and only if $I_{\Del^\alpha}$ is clean.
\end{cor}

\section{Homological invariants of the Stanley-Reisner ideal of an expanded complex}

In this section, we compare the regularity, the projective dimension and the depth of the Stanley-Reisner rings of a simplicial complex and its expansions.
The following theorem, gives an upper bound for the regularity of the Stanley-Reisner ideal of an expanded complex in terms of $\reg(I_{\Delta})$.

\begin{thm}\label{regs}
Let $\Delta$ be a simplicial complex and $\alpha=(s_1,\ldots,s_n)\in\NN^n$. Then $\reg(I_{\Delta^{\alpha}})\leq \reg(I_{\Delta})+r$, where $r=|\{i:s_i>1\}|$.
\end{thm}

\begin{proof}
It is enough to show that for any $1\leq i\leq n$,
\begin{equation}\label{rs}
 \reg(I_{\Delta^{(1,\ldots,1,s_i,1,\ldots,1)}})\leq\reg(I_{\Delta})+1.
\end{equation}
Then from the equality $\Delta^{(s_1,\ldots,s_n)}=(\Delta^{(s_1,\ldots,s_{n-1},1)})^{(1,\ldots,1,s_n)}$, we have
$$\reg(I_{\Delta^{(s_1,\ldots,s_n)}})\leq\reg(I_{\Delta^{(s_1,\ldots,s_{n-1},1)}})+1,$$
 and one can get the result by induction on $n$.

To prove (\ref{rs}), we proceed by induction on $s_i$.
First we show that the inequality $\reg(I_{\Delta'})\leq\reg(I_{\Delta})+1$ holds for any expansion of $\Delta$ obtained by duplicating a vertex. Let $X$ be the vertex set of $\Delta$, $x_i\in X$ and $\Delta'$ be an expansion of $\Delta$ by duplicating $x_i$ on the vertex set $X'=X\cup\{x'_i\}$. Then $\Delta'=\Delta \cup \langle  (F\setminus\{x_i\})\cup\{x'_i\}:\ F\in \mathcal{F}(\Delta), x_i\in F\rangle$.
For a subset $Y\subseteq X$, let $P_Y=(x_i:\ x_i\in Y)$. By \cite[Proposition 5.3.10]{villarreal}, $$I_{\Delta'}=\bigcap_{F\in \mathcal{F}(\Delta')}P_{X'\setminus F}=\bigcap_{F\in \mathcal{F}(\Delta)}(P_{X\setminus F}+(x'_i))\bigcap (\bigcap_{F\in \mathcal{F}(\Delta),x_i\in F}P_{X'\setminus ((F\setminus\{x_i\})\cup\{x'_i\})})=$$
$$((x'_i)+\bigcap_{F\in \mathcal{F}(\Delta)}P_{X\setminus F})\bigcap ((x_i)+\bigcap_{F\in \mathcal{F}(\Delta),x_i\in F}P_{X\setminus F}).$$
So
\begin{equation}\label{stanexpan}
I_{\Delta'}=(x_ix'_i)+I_{\Delta}+x'_iI_{\lk_{\Delta}(x_i)}.
\end{equation}

Let $S'=S[x'_i]$ and consider the following short exact sequence of graded modules
$$0\longrightarrow  S'/(I_{\Delta'}:x'_i)(-1)  \longrightarrow  S'/I_{\Delta'}\longrightarrow  S'/(I_{\Delta'},x'_i) \longrightarrow 0.$$
By equation (\ref{stanexpan}), one can see that  $(I_{\Delta'},x'_i)=(I_{\Delta},x'_i)$ and $(I_{\Delta'}:x'_i)=(I_{\lk_{\Delta}(x_i)},x_i)$. So we have the exact sequence
$$0\longrightarrow  (S'/(x_i,I_{\lk_{\Delta}(x_i)}))(-1) \longrightarrow  S'/I_{\Delta'} \longrightarrow  S'/(x'_i,I_{\Delta}) \longrightarrow 0.$$
Thus by \cite[Lemma 3.12, Lemma 3.5]{MoVi},
\begin{equation}\label{r0}
\reg(S'/I_{\Delta'})\leq\max\{\reg(S/I_{\lk_{\Delta}(x_i)})+1,\reg (S/I_{\Delta})\}.
\end{equation}
By \cite[Lemma 2.5]{HW}, $\reg(S/I_{\lk_{\Delta}(x_i)})\leq\reg (S/I_{\Delta})$.
So $\reg(S'/I_{\Delta'})\leq\reg (S/I_{\Delta})+1$ or equivalently $\reg(I_{\Delta'})\leq\reg (I_{\Delta})+1$.
Note that $\Delta^{(1,\ldots,1,s_i,1,\ldots,1)}$ is obtained from $\Delta^{(1,\ldots,1,s_i-1,1,\ldots,1)}$ by duplicating $x_i$. Thus by (\ref{r0}),
$$\reg(I_{\Delta^{(1,\ldots,1,s_i,1,\ldots,1)}})\leq\max\{\reg(I_{\lk_{\Delta^{(1,\ldots,1,s_i-1,1,\ldots,1)}}(x_i)})+1,\reg (I_{\Delta^{(1,\ldots,1,s_i-1,1,\ldots,1)}})\}.$$

But $\lk_{\Delta^{(1,\ldots,1,s_i-1,1,\ldots,1)}}(x_i)=\lk_{\Delta}(x_i)$.
Also by induction hypothesis $\reg(I_{\Delta^{(1,\ldots,1,s_i-1,1,\ldots,1)}})\leq\reg (I_{\Delta})+1$.
Again using \cite[Lemma 2.5]{HW}, we get the result.
\end{proof}

The next result which generalizes \cite[Theorem 3.1]{KhMo} explains $\pd(S^{\alpha}/I_{\Del^{\alpha}})$ in terms of $\pd(S/I_{\Del})$ for a sequentially Cohen-Macaulay simplicial complex.

\begin{thm}\label{pd}
Let $\Delta$ be a sequentially Cohen-Macaulay simplicial complex on $X$  and $\alpha=(s_1,\ldots,s_n)$. Then $$\pd(S^{\alpha}/I_{\Del^{\alpha}})=\pd(S/I_{\Del})+s_1+\cdots+s_n-n$$
and
$$\depth(S^{\alpha}/I_{\Del^{\alpha}})=\depth(S/I_{\Del}).$$
\end{thm}
\begin{proof}
By Corollary \ref{SCM}, $\Delta^{\alpha}$ is sequentially Cohen-Macaulay too. Thus by \cite[Corollary 3.33]{MoVi}, $$\pd(S^{\alpha}/I_{\Del^\alpha})=\bight(I_{\Del^\alpha})$$
and $$\pd(S/I_\Del)=\bight(I_\Del).$$ Let $t=\min\{|F|:\ F\in \F(\Del)\}$.  Then
$\min\{|F|:\ F\in \F(\Del^{\alpha})\}=t$, $\bight(I_{\Del})=n-t$ and
$$\bight(I_{\Del^{\alpha}})=|X^{\alpha}|-t=s_1+\cdots+s_n-t=s_1+\cdots+s_n+\pd(S/I_\Del)-n.$$
The second equality holds by Auslander-Buchsbaum formula. Note that $\depth(S^\alpha)=s_1+\cdots+s_n$.
\end{proof}

Let $\fm$ and $\fn$ be, respectively, the maximal ideals of $S$ and $S^\alpha$.

\begin{thm}\label{localco}
(Hochster \cite{HeHi}) Let $\ZZ^n_{-}=\{\a\in\ZZ^n:a_i\leq 0\ \mbox{for}\ i=1,\ldots,n\}$. Then
$$\dim H^i_\fm(K[\Del])_\a=\left\{
                     \begin{array}{ll}
                      \dim\tilde{H}_{i-|F|-1}(\lk_\Del (F);K) , & \hbox{if}\ \a\in\ZZ^n_{-},\ \hbox{where}\ F=\supp(\a) \\
                      0, & \hbox{if}\ \a\not\in\ZZ^n_{-}.
                     \end{array}
                   \right.
$$
\end{thm}

\begin{thm}
Let $\Del$ be a simplicial complex on $X$ and $\alpha=(s_1,\ldots,s_n)$. Then $\depth(K[\Del^\alpha])\leq\depth(K[\Del])$.
It follows that $\pd(S^{\alpha}/I_{\Del^{\alpha}})\geq\pd(S/I_{\Del})+s_1+\cdots+s_n-n$.
\end{thm}
\begin{proof}
Set $s=\sum_is_i$. For $\a\in\ZZ^k$, set $s_0=0$ and for $1\leq i\leq n$ set $\bar{\a}(i)=\sum^{s_i}_{j=s_{i-1}+1}\a(j)$. Let $\a\in\ZZ^s_-$ and $F=\supp(\a)$. It follows from Theorem \ref{localco} that $\dim H^i_\fn(K[\Del^\alpha])_\a=\dim\tilde{H}_{i-|F|-1}(\lk_{\Del^\alpha}(F);K)$.

On the other hand by \cite[Theorem 6.2.7]{BrSh} $\depth(K[\Del])$ is the least integer $i$ such that $H^i_\fm(K[\Del])\neq 0$. Now, if $\depth(K[\Del^\alpha])=d$, then $\tilde{H}_{d-|F|-1}(\lk_{\Del^\alpha}(F);K)\neq 0$ and $\tilde{H}_{i-|F|-1}(\lk_{\Del^\alpha}(F);K)\neq 0$ for any $i<d$. By Proposition \ref{expansion}, one can see that
$\dim\tilde{H}_{i-|\bar{F}|-1}(\lk_\Del(\bar{F});K)=\dim H^i_\fm(K[\Del])_{\bar{\a}}=0$ for any $i<d$. This obtains the assertion.

The second inequality holds by Auslander-Buchsbaum formula.
\end{proof}

\textbf{Acknowledgments:}
The authors would like to thank the referee for careful reading of the paper. The research of Rahim Rahmati-Asghar and Somayeh Moradi were in part supported by a grant from IPM (No. 94130029) and (No. 94130021).

\ \\ \\
Rahim Rahmati-Asghar,\\
Department of Mathematics, Faculty of Basic Sciences,
University of Maragheh, P. O. Box 55181-83111, Maragheh, Iran and School of Mathematics, Institute
for Research in Fundamental Sciences (IPM), P.O.Box 19395-5746, Tehran, Iran.\\
E-mail:  \email{rahmatiasghar.r@gmail.com}
\ \\ \\
Somayeh Moradi,\\
Department of Mathematics, Ilam University, P.O.Box 69315-516,
Ilam, Iran and School of Mathematics, Institute
for Research in Fundamental Sciences (IPM), P.O.Box 19395-5746, Tehran, Iran.
E-mail: \email{somayeh.moradi1@gmail.com}

\end{document}